\documentclass[11pt,reqno,oneside]{amsart}
\usepackage{graphics}
\usepackage[dviwin]{color}
\def\Box{\setlength{\unitlength}{0.01cm}
     \begin{picture}(15,15)(-5,-5)
       \framebox(15,15){}
     \end{picture} }

\newtheorem{theorem}{Theorem}[section]
\newtheorem{lemma}{Lemma}[section]
\newtheorem{corollary}{Corollary}[section]
\newtheorem{remark}{Remark}[section]
\newtheorem{defi}{Definition}[section]

\def\o{\Omega}
\def\g{\gamma}
\def\a{\alpha}
\def\b{\beta}
\def\bs{\bigskip}

\def\m{\mu}

\def\uu{{\bf u}}
\def\L{{\bf L}}
\def\vv{{\bf v}}
\def\ff{{\bf f}}
\def\ww{{\bf w}}
\def\rr{{\bf r}}
\def\Z{{\mathbb {Z}}}
\def\R{{\mathbb R}}

\def\W1p{W^{1,p}(\o)}

\def\di{\mbox{div\,}}

\def\ve{\varepsilon}

\def\cu{\mbox{{\bf curl}\,}}

\begin{document}
\section*{}

\setcounter{equation}{0}
\title[]{A right inverse of the divergence for\\ planar H\"older-{\LARGE
$\alpha$} domains}

\author[R. G. Dur\'an]{Ricardo G. Dur\'an}
\address{Departamento de Matem\'atica\\ Facultad de
Ciencias Exactas y Naturales\\Universidad de Buenos Aires\\1428 Buenos Aires\\
Argentina.}
\email{rduran@dm.uba.ar}

\author[F. L\'opez Garc\'\i a]{Fernando L\'opez Garc\'\i a}
\address{Departamento de Matem\'atica\\ Facultad de
Ciencias Exactas y Naturales\\Universidad de Buenos Aires\\1428 Buenos Aires\\
Argentina.}
\email{flopezg@dm.uba.ar}



\begin{abstract} If $\Omega\subset\R^n$ is a bounded domain, the existence
of solutions ${\bf u}\in H^1_0(\Omega)^n$
of $\mbox{div\,} {\bf u} = f$ for $f\in L^2(\Omega)$
with vanishing mean value, is a basic result in the analysis of the Stokes
equations.
In particular it allows to show the existence of a solution
$({\bf u},p)\in H^1_0(\Omega)^n\times L^2(\Omega)$, where ${\bf u}$ is the velocity and $p$ the pressure.

It is known that the above mentioned result holds when $\Omega$ is a Lipschitz
domain and that it is not valid for arbitrary H\"older-$\alpha$ domains.

In this paper we prove that if $\Omega$ is a planar simply connected
H\"older-$\alpha$ domain, there exist right inverses of the
divergence which are continuous in appropriate weighted spaces,
where the weights are powers of the distance to the boundary.
Moreover, we show that the powers of the distance in the results
obtained are optimal.

In our results, the zero boundary condition is replaced by a weaker
one. For the particular case of domains with an external cusp of
power type, we prove that our weaker boundary condition
is equivalent to the standard one.
In this case we show the well posedness of the Stokes equations
in appropriate weighted Sobolev spaces obtaining as a consequence
the existence of a solution $({\bf u},p)\in H^1_0(\Omega)^n\times L^r(\Omega)$
for some $r<2$ depending on the power of the cusp.
\end{abstract}
\maketitle

\section{Introduction}
\label{intro}
\setcounter{equation}{0}

Let $\o\in\R^n$ be a bounded open domain. We will use standard notations for
Sobolev spaces and, for $1<p<\infty$, $L_0^p(\o)$ will denote the subspace
of functions in $L^p(\o)$ with vanishing mean value.

The existence of right inverses
of the operator $\di : H_0^1(\o)^n\to L_0^2(\o)$
is a basic result for the theoretical and numerical analysis of the Stokes
equations
by variational methods (see for example \cite{BS,BF,GR,T}). This result is also
closely connected
with the Korn inequality which is fundamental in the analysis of the elasticity
equations
(see \cite{HP}).

Usually, the problem is stated as follows: for any $f\in L_0^2(\o)$ there exists
$\uu\in H_0^1(\o)^n$ such that

\begin{equation}
\di\uu=f \qquad \mbox{in}\ \ \o
\label{divu=f}
\end{equation}
and
\begin{equation}
\|\uu\|_{H_0^1} \le C \|f\|_{L^2}
\label{apriori}
\end{equation}
where, here and throughout the paper, the letter $C$ denotes a
generic constant.

This problem, as well as its generalization to the $L^p$ case, has been
widely analyzed and several different arguments have been given to prove it
under different assumptions on the domain. We refer the reader for example
to \cite{ASV,BB,BS,B,GR,DM2}. Recently, in \cite{ADM}, the existence of right
inverses
of the divergence acting on $W_0^{1,p}(\o)^n$, $1<p<\infty$, was proved for
the so called John domains, which form a large class containing properly
the Lipschitz domains. Moreover, for the particular case of planar simply
connected
domains, it is shown in \cite{ADM} that being a John domain is also a necessary
condition in the case $1<p<2$.
In particular, there exist
bounded domains and values of $p$ for which continuous right inverses
of the divergence do not exist.

Actually, this fact was previously well known, indeed,
several arguments have been given to show it. For example,
in the old paper \cite{F1}, Friedrichs
proved that, for smooth planar domains, the $L^2$-norm of the conjugate of a
harmonic
function $f$ (normalized in an appropriate way) is bounded by the $L^2$-norm of
$f$ times a constant depending only on the domain. Moreover, he showed
that this inequality is not valid if the domain has an external cusp of
quadratic type.
It is easy to see that the Friedrichs inequality can be deduced from the
existence of $\uu$ satisfying (\ref{divu=f}) and (\ref{apriori}). Therefore,
such
a $\uu$ cannot exist for that kind of domains.
More recently other examples have been given in \cite{GG} and in an unpublished
work of Gabriel Acosta. Acosta's examples are very elementary
and applies to external
cusps of power type with any power $\g>1$ and any $1<p<\infty$. We
also refer the reader
to \cite{D1} where a particular case has been reproduced and to \cite{ADL}.

In view of the above mentioned results, it seems natural to ask
whether some weaker results can be proved for a more general class of domains.
Moreover, if this is the case,
can those results be applied to show the well posedness of the Stokes problem
in appropriate Hilbert spaces?

In this paper we give some partial answers to these questions in the particular
case of planar simply connected domains. We consider H\"older $\a$ domains,
with $0<\a\le 1$. We say that a domain belongs to this class if
its boundary is locally the graph of a H\"older $\a$ function.
For these domains we prove the existence of solutions of
(\ref{divu=f}) satisfying weaker estimates than (\ref{apriori})
involving weighted norms where the weights are powers of
the distance to the boundary. For general H\"older $\a$ domains
the zero boundary condition will be imposed in a weak way. Afterwards, in some
particular
examples, we will show that this weak boundary condition agrees with
the usual one.

Our approach use some of the ideas of the papers \cite{ADL,GK}.
The existence of solutions of the divergence is derived from appropriate Korn
type inequalities. The weighted Korn inequalities that we need are slight
variants
of those obtained in \cite{ADL} but we include the proofs for the sake of
completeness.

Although our arguments to derive
the existence of right inverses of the divergence are two dimensional,
we write the proofs of the Korn type inequalities in the general $n$-dimensional
case because they have interest in themselves.

The rest of the paper is organized as follows.
In Section \ref{korn} we introduce some notations and prove the weighted
Korn inequalities. Section \ref{right inverse} deals with our main results
concerning the existence of right inverses of the divergence continuous in
appropriate weighted norms for H\"older $\a$ domains. In Section
\ref{external cusps}
we apply the results of the previous section for the particular case
of domains having power type external cusps. We show that in this case
our weak zero boundary condition agrees with the usual one. Also in this section
we prove optimality of our results. In Section \ref{application} we show how our
results can be applied to prove the well posedness of the Stokes equations in
appropriate
Hilbert spaces.
Recall that, if $\uu$ is the velocity and $p$ the pressure of a viscous
incompressible
fluid, the Stokes equations are given by
\begin{align}
\label{stokes}
\begin{cases}
-\Delta\uu\,+\,\nabla p&=\,\ff\hspace*{1cm}{\rm in}\ \o\\
\di\uu&=\,0\hspace*{1cm}{\rm in}\ \o\\
\uu&=\,0\hspace*{1cm}{\rm in}\ \partial\o.
\end{cases}
\end{align}
The existence of solutions of (\ref{divu=f}) satisfying (\ref{apriori}) can be
used to prove that, for $\ff\in H^{-1}(\o)^2$, there exists a unique solution
$(\uu,p)\in H^1_0(\o)^2 \times L_0^2(\o)$ of (\ref{stokes}) and moreover,
$$
\|\uu\|_{H^1_0(\o)} + \|p\|_{L^2(\o)}
\le C \|\ff\|_{H^{-1}(\o)}.
$$
We will show that our right inverses of the divergence can be used
to prove a slightly weaker result for some cuspidal domains.
In particular, under some restriction on the power of
the cusp, we will obtain the existence of a unique solution
$(\uu,p)\in H^1_0(\o)^2 \times L_0^r(\o)$
of the Stokes equations (\ref{stokes}) satisfying
$$
\|\vv\|_{H^1_0(\o)} + \|p\|_{L^r(\o)}
\le C \|\ff\|_{H^{-1}(\o)}
$$
for some $1<r\le 2$ which depends on the power of the cusp.

\section{Preliminaries and Korn type inequalities}
\setcounter{equation}{0}
\label{korn}

Let $\o$ be a bounded open subset of $\R^n$ and $d(x)$
the distance of $x\in\o$ to the boundary $\partial\o$.
We will denote by $L^p(\o,\g)$ the Banach space given by
the norm
$$
\|u\|_{L^p(\o,\g)}:=\|u\,d^\g\|_{L^p(\o)}
$$
and, analogously, $W^{1,p}(\o,\gamma)$ will be the Banach space
with norm
\begin{eqnarray}
\label{defiw1p}
\|u\|_{W^{1,p}(\o,\g)}:=\|u\,d^\g\|_{L^p(\o)}\,
+\,\|\nabla u \,d^\g\|_{L^p(\o)}.
\end{eqnarray}
Whenever $L^p(\o,\g)\subset L^1(\o)$ we will call $L_0^p(\o,\g)$
the subspace of $L^p(\o,\g)$ formed by functions of vanishing mean value.
Since no confusion is possible we will use the same notations for
the norms of vector or tensor fields.

For a vector field $\uu=(u_1,\cdots,u_n)$ defined in $\o$ we denote by
$D\uu$ the jacobian matrix, namely, $(D\uu)_{ij}=\frac{\partial u_i}{\partial
x_j}$
and by $\ve(\uu)$ its symmetric part (i.e., the linear strain tensor associated
with $\uu$),
that is
$\ve(\uu)_{ij}=\frac12 \left(\frac{\partial u_i}{\partial x_j}
+\frac{\partial u_j}{\partial x_i}\right)$.

\bigskip

We start by giving a weighted Korn inequality for
H\"older $\alpha$ domains. The statement given in the following
theorem is slightly stronger than the result in Theorem 3.1 of \cite{ADL}.
Therefore, we include the proof for the sake of completeness although
the arguments are essentially those given in that reference. In particular
we will make use of the following improved Poincar\'e inequality proved
in \cite[Theorem 2.1]{ADL}.
If $\o$ is a H\"older $\alpha$ domain, $0<\a\le 1$, $B\subset\o$
a ball and $\phi\in C_0^\infty(B)$ is such that $\int_B \phi=1$ then,
for $\a\le\b\le 1$ and $f$ such that $\int_B f\phi=0$
there exists a
constant $C$ depending only on $\o$, $B$ and $\phi$ such that,
\begin{equation}
\label{improved poincare}
\|f\|_{L^p(\o,1-\beta)}
\le C\,\|\nabla f\|_{L^p(\o, 1+\alpha-\beta)}.
\end{equation}

\bigskip

\begin{theorem}
\label{kornpesos}
Let $\o\subset \R^n$ be a H\"older $\a$
domain, $B\subset\o$ a ball and $1<p<\infty$. Then, for
$\a\le\b\le 1$ the following inequality holds,
$$
\|D\uu\|_{L^p(\o,1-\beta)}\,\leq\,
C\Bigl\{\|\ve(\uu)\|_{L^p(\o,\alpha-\beta)}+\|\uu\|_{L^p(B)}\Bigr\}
$$
where the constant $C$ depends only on $\o$, $B$ and $p$.
\end{theorem}

\begin{proof} Following \cite{KO}, we can show that there exists $\vv\in
W^{1,p}(\o)^n$
such that
\begin{eqnarray}
\label{lap}
\Delta\vv=\Delta\uu \hspace*{1cm}{\rm in}\ \o
\end{eqnarray}
and
\begin{eqnarray}
\label{lapn} \|\vv\|_{W^{1,p}(\o)}\,\le\,C\,\|\ve(\uu)\|_{L^p(\o)}.
\end{eqnarray}

Now, let $\phi\in C_0^\infty(B)$ be such that $\int_B\phi\, dx=1$.
For $i=1,...,n$ define the linear functions
$$
L_i(x):=\left(\int_B\nabla(u_i-v_i)\phi\right)\cdot x
$$
and $\L(x)$ as the vector with components $L_i(x)$.

Then,
$$
D\L=\int_B D(\uu-\vv)\phi
$$
and, integrating by parts and applying the H\"older inequality we obtain
$$
|D\L|\leq\|\uu-\vv\|_{L^p(B)}\|\nabla \phi\|_{L^{p'}(B)}
$$
where $p'$ is the dual exponent of $p$.

Therefore, it follows from (\ref{lapn}) that there exists a
constant $C$ depending only on $\o$, $p$ and $\phi$ such that

\begin{equation}
\label{korn1}
\|D\L\|_{L^p(\o)}
\le C\,\Bigl\{\|\uu\|_{L^p(B)}+\|\ve(\uu)\|_{L^p(\o)}\Bigr\}.
\end{equation}
Let us now introduce
$$
\ww:= \uu-\vv-\L.
$$
Then, in view of the bounds
(\ref{lapn}) and (\ref{korn1}), it only remains to estimate $\ww$.
But, from (\ref{lap}) and the fact that $\L$ is
linear we know that
$$\Delta \ww=0$$
and consequently,
$$
\Delta \ve_{ij}(\ww)=0.
$$
But, if $f$ is a harmonic function in $\o$, the following estimate holds
$$
\|\nabla f\|_{L^p(\o,1-\m)} \le C \|f\|_{L^p(\o,-\m)}
$$
for all $\m\in\R$. Indeed, this estimate was proved in \cite{D} (see also
Lema 3.1 in \cite{ADL}, and \cite{KO} for a different proof in the case
$p=2$ and $\m=0$).

Therefore, taking $\mu=\beta-\alpha$ we obtain
$$
\|\nabla\ve_{ij}(\ww)\|_{L^p(\o,1+\a-\beta)}
\le C\|\ve_{ij}(\ww)\|_{L^p(\o,\a-\beta)}
$$
and using the well
known identity

$$
\frac{\partial^2w_i}{\partial x_j\partial x_k}\,=\,
\frac{\partial \ve_{ik}(\ww)}{\partial x_j}\,
+\, \frac{\partial\ve_{ij}(\ww)}{\partial x_k}\,
-\,\frac{\partial\ve_{jk}(\ww)}{\partial x_i}
$$
we conclude that

\begin{eqnarray}
\label{korn2,5} \left\|\frac{\partial^2w_i}{\partial x_j\partial
x_k}\right\|_{L^p(\o,1+\alpha-\beta)}
\leq C\|\ve(\ww)\|_{L^p(\o,\alpha-\beta)}
\end{eqnarray}
for any $i,j$ and $k$.

Since $\int \frac{\partial w_i}{\partial x_j}\phi=0$ (indeed, we have defined
$\L$ in
order to have this property), it follows from the improved Poincar\'e
inequality (\ref{improved poincare}) that
$$
\left\|\frac{\partial w_i}{\partial x_j}\right\|_{L^p(\o,1-\beta)}
\leq C\,\left\|\nabla\frac{\partial w_i}{\partial x_j}\right\|_{L^p(\o,
1+\alpha-\beta)}.
$$
Therefore, using (\ref{korn2,5}), we obtain
\begin{eqnarray*}
\|D\ww\|_{L^p(\o,1-\beta)}\,\leq\,
C\|\ve(\ww)\|_{L^p(\o,\alpha-\beta)}\,\leq\,C\|\ve(\uu)\|_{L^p(\o,\alpha-\beta)}
\end{eqnarray*}
concluding the proof.
\end{proof}

In the following corollary we give a weighted Korn inequality for H\"older
$\a$ domains which can be seen as a generalization of the so-called
second case of Korn inequality. To state this inequality we need to introduce
the space of infinitesimal rigid motions, namely,

$$
{\mathcal N}=\{\vv\in W^{1,p}(\o)^n\, : \, \ve(\vv)=0\}.
$$

\begin{corollary}
\label{kornpeso} Let $\o\subset\R^n$ be a H\"older
$\alpha$ domain and $1<p<\infty$. Then, for $\alpha\leq\beta\le 1$ the
following inequality holds,
\begin{equation}
\label{korn2}
\inf_{\vv\in {\mathcal N}}\|\uu-\vv\|_{W^{1,p}(\o,1-\beta)}\,
\le\,C\|\varepsilon(\uu)\|_{L^p(\o,\alpha-\beta)}.
\end{equation}
\end{corollary}

\begin{proof} Take $B$ and $\phi$ as in the previous theorem with
$\overline B\subset\o$. Define $\overline x_i=\int_{B}x_i\phi(x)\,dx$
and $\vv\in W^{1,p}(\o)^n$ defined by
$$
v_i(x)=a_i + \sum_{j=1}^n b_{ij}(x_j-\overline x_j)
$$
with
$$
a_i=\int_B u_i\phi
\qquad
\mbox{and} \qquad
b_{ij}=\frac1{2|B|}\int_{B}
\left(\frac{\partial u_i}{\partial x_j}
-\frac{\partial u_j}{\partial x_i}\right).
$$
It is easy to check that $\vv\in{\mathcal N}$.
Now, since $\int_B (\uu-\vv)\phi=0$, it follows from
(\ref{improved poincare}) (actually we are using only a weaker
standard Poincar\'e inequality with weights) and Theorem \ref{kornpesos}
that
$$
\|\uu-\vv\|_{W^{1,p}(\o,1-\beta)}\,
\le\, C\Bigl\{\|\ve(\uu-\vv)\|_{L^p(\o,\alpha-\beta)}
+\|\uu-\vv\|_{L^p(B)}\Bigr\}
$$
and using now the Poincar\'e inequality in $B$ we have
\begin{equation}
\label{casikorn}
\|\uu-\vv\|_{W^{1,p}(\o,1-\beta)}\,
\le\, C\Bigl\{\|\ve(\uu-\vv)\|_{L^p(\o,\alpha-\beta)}
+\|D(\uu-\vv)\|_{L^p(B)}\Bigr\}.
\end{equation}
But,
$$
\int_{B}
\left(\frac{\partial (u-v)_i}{\partial x_j}
-\frac{\partial (u-v)_j}{\partial x_i}\right)
=0
$$
and therefore, the so-called second case of Korn
inequality applied in $B$ gives
$$
\|D(\uu-\vv)\|_{L^p(B)}
\le C \|\ve(\uu-\vv)\|_{L^p(B)}.
$$
Using this inequality in (\ref{casikorn}) and
that $\ve(\vv)=0$ we obtain
$$
\|\uu-\vv\|_{W^{1,p}(\o,1-\beta)}\,
\le\, C\Bigl\{\|\ve(\uu)\|_{L^p(\o,\alpha-\beta)}
+\|\ve(\uu)\|_{L^p(B)}\Bigr\}
$$
which implies (\ref{korn2}) because $\overline B\subset\o$. \end{proof}

\begin{remark} It is possible to prove the above corollary
directly, i.e., without using the Korn inequality in the
ball $B$, by using a standard compactness argument. Indeed,
assuming that (\ref{korn2}) does not hold and using
that $W^{1,p}(\o,1-\beta)$ is compactly embedded
in $L^p(\o,\gamma)$ for any $\gamma>(1-\beta-\alpha)/\alpha$
(see  \cite[Theorem 19.11]{KuOp}) and Theorem \ref{kornpesos}
one obtains a contradiction.
\end{remark}

\section{Right inverse of the divergence in H\"older $\a$ domains}
\setcounter{equation}{0}
\label{right inverse}

This section deals with solutions of divergence
in planar simply connected  H\"older $\a$ domains.
In what follows we restrict ourselves to the case $n=2$.

For regular enough bounded domains $\o$ (for example Lipschitz) it is known
that,
if $f\in L_0^p(\o)$, $1<p<\infty$, there exists
$\uu\in W_0^{1,p}(\o)^2$ such that
\begin{equation}
\label{divu=f1}
div\,\uu=f
\end{equation}
and
\begin{equation}
\label{estimacion}
\|\uu\|_{W_0^{1,p}(\o)}\le C \|f\|_{L^p(\o)}
\end{equation}
where the constant $C$ depends only on $\o$ and $p$.

On the other hand, as we have mentioned in the introduction,
it is known that for general H\"older $\a$ domains
this result is not valid. Our main goal is to prove a similar
result for this kind of domains but using weighted norms.

We will use the following notation. For a scalar function
$\psi$ we write
$\cu\psi=(\frac{\partial\psi}{\partial
x_2},-\frac{\partial\psi}{\partial x_1})$ and for a vector
field $\Psi=(\psi_1,\psi_2)$, $Curl\,\Psi$ denotes
the matrix which has $curl\,\psi_i$ as it rows. Furthermore,
if $\sigma\in L^p(\o)^{2\times 2}$, $Div\,\sigma$ denotes the
vector field with components obtained by taking the divergence of the
rows of $\sigma$.

We will impose the boundary condition in a weak form. To explain
this weak condition observe first that to solve the problem
it is enough to find a solution
$\uu$ of (\ref{divu=f1}) such that the restriction to $\partial\o$
of both components of $\uu$ are constant (whenever the domain is such
that this restriction makes sense). Of course, we should replace the
estimate (\ref{estimacion}) by
\begin{equation}
\label{estimacion2}
\|D\uu\|_{L^p(\o)}\le C \|f\|_{L^p(\o)}.
\end{equation}
Afterwards, (\ref{estimacion}) would follow by applying the Poincar\'e
inequality to the solution obtained by adding an appropriate constant
vector field to $\uu$ in order to obtain the vanishing boundary condition.

Now, assume that $\o$ is a Lipschitz domain. Then, if $\psi\in W^{1,p}(\o)$
satisfies
\begin{equation}
\label{debil1}
\int_\o \cu\psi\cdot\nabla\phi=0 \qquad \forall \phi\in W^{1,p'}(\o)
\end{equation}
it follows by integration by parts that
\begin{equation}
\label{debil2}
\int_{\partial\o}\frac{\partial\psi}{\partial t}\phi=0
\qquad \forall \phi\in W^{1,p'}(\o)
\end{equation}
where $\frac{\partial\psi}{\partial t}$ indicates the tangential derivative of
$\psi$.
Therefore $\frac{\partial\psi}{\partial t}=0$ and then the restriction of $\psi$
to $\partial\o$ is constant.

For a general domain $\o$ the tangential derivative on the boundary might not
even be defined and therefore (\ref{debil2}) would not make sense.
However, condition (\ref{debil1}) is well defined in any domain
and this is the condition that we will use. Therefore we introduce
the space
$$
W_{const}^{1,p}(\o)\subset W^{1,p}(\o)
$$
defined by
$$
W_{const}^{1,p}(\o)=\left\{\psi\in W^{1,p}(\o)\,:\,
\int_\o \cu\psi\cdot\nabla\phi=0 \qquad \forall \phi\in W^{1,p'}(\o)\right\}
$$
and more generally, for any $\g\in\R$,
$$
W_{const}^{1,p}(\o,\g)=\left\{\psi\in W^{1,p}(\o,\g)\,:\,
\int_\o \cu\psi\cdot\nabla\phi=0 \qquad \forall \phi\in W^{1,p'}(\o,-\g)\right\}.
$$

The proof of the following lemma uses ideas introduced in \cite{GK}
with different goals.

For $1<p<\infty$ and $\g\in\R$, $L_{sym}^p(\o,\g)^{2\times 2}$ denotes
the subspace of symmetric tensors in $L^{p}(\o,\g)^{2\times 2}$.

\begin{lemma}
\label{tensorsimetrico}
Let $\o\subset\R^2$ be a H\"older $\a$
domain and $\uu\in W^{1,p}(\o,\beta-1)^2$, with $\alpha\le\beta\le 1$,
such that $\int_\o \di\uu=0$. Then, there exists $\sigma\in
L_{sym}^p(\o,\beta-\alpha)^{2\times 2}$
satisfying
$$
\int_{\o}\sigma:D\,\ww\,=\,\int_\o Curl\,\uu:D\ww,
\qquad \forall \ww\in W^{1,p'}(\o,\alpha-\beta)^2
$$
and
$$
\|\sigma\|_{L^p(\o,\beta-\alpha)^{2\times 2}}
\le C\|Curl\,\uu\|_{L^p(\o,\beta-1)^{2\times 2}}.
$$
\end{lemma}

\begin{proof} Let $H\subset L_{sym}^{p'}(\o,\alpha-\beta)^{2\times 2}$ the
subspace
defined as
$$
H=\{\tau\in L_{sym}^{p'}(\o,\alpha-\beta)^{2\times 2}
\,:\,\tau=\ve(\ww)\,\, \mbox{with}\ \ww\in W^{1,p'}(\o,\alpha-\beta)^2\}.
$$
Let us see that the application
\begin{equation}
\label{funcional}
T\, : \, \ve(\ww)\mapsto\int_\o Curl\,\uu:D\ww
\end{equation}
defines a continuous linear functional on $H$.

First of all observe that $T$ is well defined. Indeed,
it is enough to check that the expression on the right of (\ref{funcional})
vanishes whenever $\ve(\ww)=0$. But, it is known that in that case
$\ww(x,y)=(a-cy,b+cx)$
and therefore
$$
\int_\o Curl\,\uu:D\ww = c\int_\o \di\uu=0.
$$
Now, we want to show that $T$ is continuous on $H$. Using again that
$\int_\o Curl\,\uu:D\vv=0$ if $\ve(\vv)=0$ and applying
Corollary \ref{kornpeso} we have, for $\tau=\ve(\ww)\in H$,

\begin{eqnarray*}
|T(\tau)|\,&=&\,\left|\int_\o Curl\,\uu:D\ww \right|\\
&\le&\,\|Curl\,\uu\|_{L^p(\o,\beta-1)^{2\times 2}}\,\inf_{\vv\in
{\mathcal N}}\|
D(\ww-\vv)\|_{L^{p'}(\o,1-\beta)^{2\times 2}}\\
&\le& C\|Curl\,\uu\|_{L^p(\o,\beta-1)^{2\times 2}}
\|\ve(\ww)\|_{L^{p'}(\o,\alpha-\beta)^{2\times 2 }}\\
&=& \,C\|Curl\,\uu\|_{L^p(\o,\beta-1)^{2\times 2}}
\|\tau\|_{L^{p'}(\o,\alpha-\beta)^{2\times 2 }.}
\end{eqnarray*}

By the Hahn-Banach theorem the functional $T$
can be extended to $L_{sym}^{p'}(\o,\alpha-\beta)^{2\times 2}$
and therefore, by the Riesz representation theorem, there exists
$\sigma\in L_{sym}^p(\o,\beta-\alpha)^{2\times 2}$
such that

\begin{eqnarray*}
T(\tau)\,=\,\int_\o \sigma:\tau \qquad \forall \tau\in
L_{sym}^{p'}(\o,\alpha-\beta)^{2\times 2}
\end{eqnarray*}
and
$$
\|\sigma\|_{L^p(\o,\beta-\alpha)^{2\times 2}}
\le C\|Curl\,\uu\|_{L^p(\o,\beta-1)^{2\times 2}},
$$
where $C$ depends on the constant in Corollary \ref{kornpeso}.
In particular,

\begin{eqnarray}
\label{parte4}
\int_{\o}\sigma:\ve(\ww)\,=\,\int_\o Curl\,\uu:D\ww
\end{eqnarray}
for every $\ww\in W^{1,p'}(\o,\alpha-\beta)^2$.
Then, we conclude the proof observing that, since $\sigma$ is symmetric,
we can replace $\ve(\ww)$
in (\ref{parte4}) by $D\ww$. \end{proof}

It is a very well known result that a divergence free vector
field is a rotational of a scalar function $\phi$. Indeed, for smooth vector
fields the
proof is usually given at elementary courses on calculus in several variables.
On the other hand, if the vector field is
only in $L^p(\o)^2$ but $\partial\o$ is Lipstchiz, it is not
difficult to see that the vector field can be extended to a divergence
free vector field defined in $\R^2$ and then, the existence of $\phi$
can be proved by using the Fourier transform.
However, we need to use the existence of $\phi$ in the case where the domain and
the vector field
are both non-smooth. We have not
been able to find a proof of this result in the literature and so we include the
following lemma.

\begin{lemma}
\label{rotor}
Let $\o\subset\R^2$ be a simply connected H\"{o}lder
$\alpha$ domain and $\alpha \leq\beta\leq 1$. Given a vector
field $\vv\in L^p(\o,1-\beta)^2$ such that $\di\vv=0$, there
exists $\phi\in W^{1,p}(\o,1-\beta)$ such that
$$
\cu\phi=\vv
\hspace*{1cm}{\rm and}
\hspace*{1cm}
\|\phi\|_{W^{1,p}(\o,1-\beta)}
\le C\|\vv\|_{L^p(\o,1-\beta)}
$$
where $C$ is a constant depending only on $\o$.
\end{lemma}
\begin{proof}
Take $\psi\in C_0^{\infty}(B_1)$ satisfying
$\int\psi=1$, where $B_1$ is the unit ball centered at
the origin. For $k\ge 1$, define $\psi_k(x)=k^2\psi(kx)$
and, extending $\vv$ by zero to $\R^2$, $\vv_k=\psi_k\ast\vv$.

Let $\Omega_n$ be a sequence of Lipschitz simply
connected open subsets of $\o$ such that
$$
\overline\o_n\subset\Big\{x\in\o :\,d(x)>\frac{1}{n}\Big\}
\hspace*{1cm}{\rm and} \hspace*{1cm}\o_n\nearrow\o.
$$
Using that the distance between
$\o_n$ and $\partial\o$ is greater than $1/n$ and
${\rm supp}\,\psi_k\subset B(0,\frac{1}{k})$, it is not
difficult to see that
$\di\vv_k=0$ in $\o_n$ for every $k\ge n$.

Then, since $\vv_n\in C^\infty_0(\R^2)^2$, there exists
$\phi_n\in C^\infty_0(\o_n)$ such that $\cu\phi_n=\vv_n$.
Moreover, adding a constant we can take $\phi_n$
such that $\int_{\o_1} \phi_n=0$.

Now, by the Poincar\'e inequality we have, for any $n$,
there exists a constant $C$ depending only on $n$ such that
$$
\|\phi_k-\phi_{k^\prime}\|_{L^p(\o_n)}
\le C \|\cu(\phi_k-\phi_{k^\prime})\|_{L^p(\o_n)}
= C \|\vv_k-\vv_{k^\prime}\|_{L^p(\o_n)}\rightarrow 0
$$
for $k,k^{\prime}\to\infty$.

Then, there exists $\phi\in L^1_{loc}(\o)$ such that
$\phi_k|_{\o_n}\to \phi$ in $W^{1,p}(\o_n)$ and so
$\cu\phi=\vv$ in $\o_n , \forall n$ and consequently in $\o$.

Finally, using Theorem 2.1 of \cite{ADL} we have
$$
\|\phi\|_{L^p(\o,1-\beta)}
\le C\|\cu\phi\|_{L^p(\o,1-\beta+\a)}
\le C\|\vv\|_{L^p(\o,1-\beta)}
$$
and the Lemma is proved.
\end{proof}

We can now state and prove our results on solutions of the divergence
on H\"older-$\a$ domains. As we mentioned above, it is known that
for this kind of domains a solution of (\ref{divu=f1}) satisfying
(\ref{estimacion})
does not exist in general. Therefore, it is natural to look for solutions
of (\ref{divu=f1}) satisfying a weaker estimate. There are two possibilities:
to use a stronger norm on the right of (\ref{estimacion}) or a weaker
norm on the left. We will prove both kind of results but, to avoid technical
complications
while presenting the arguments, we give first a particular case of our results
and postpone the generalization.

\begin{theorem}
\label{divergencia-part}
Let $\o\subset\R^2$ be a bounded simply connected
H\"older-$\a$ domain, $0<\a\le 1$.
Given $f\in L_0^p(\o)$, $1<p<\infty$, there
exists
$\uu\in W_{const}^{1,p}(\o,1-\alpha)^2$ such that
$$
\di\uu=f
$$
and
\begin{equation}
\label{pesoalaizq}
\|D\uu\|_{L^p(\o,1-\alpha)}
\le C\|f\|_{L^p(\o)}
\end{equation}
\end{theorem}

\begin{proof} Take $\vv\in W^{1,p}(\o)^2$ such that
\begin{equation}
\label{casopart1}
\di\vv=f
\end{equation}
and
\begin{equation}
\label{casopart2}
\|\vv\|_{W^{1,p}(\o)} \le C\|f\|_{L^p(\o)}.
\end{equation}
The existence of such a $\vv$ is well known, for example,
since no boundary condition on $\vv$ is required,
we can extend $f$ by zero and take the solution of problem
(\ref{divu=f1}) and (\ref{estimacion})
in a ball containing $\o$.

To prove the theorem it is enough to
show that there exists $\ww\in W^{1,p}(\o,1-\alpha)^2$
satisfying $\di\ww=0$ and such that
$$
\vv-\ww \in W_{const}^{1,p}(\o,1-\alpha)^2
$$
and
\begin{equation}
\label{casopart3}
\|D\ww\|_{L^p(\o,1-\alpha)}
\le C\|f\|_{L^p(\o)}.
\end{equation}
Indeed, in view of (\ref{casopart1}),
$\uu:=\vv-\ww$ will be the desired solution.

But, since $\di\vv$ has vanishing mean value, we know from
Lemma \ref{tensorsimetrico} that there exists
$\sigma\in L_{sym}^p(\o,1-\alpha)^{2\times 2}$ satisfying
\begin{equation}
\label{casopart4}
\|\sigma\|_{L^p(\o,1-\alpha)}
\le C\,\|Curl\,\vv \|_{L^p(\o)}
\end{equation}
and
$$
\int_{\o}\sigma:D\rr = \int_\o Curl\,\vv:D\rr
\quad , \quad\forall\rr\in W^{1,p'}(\o,\alpha-1)^2.
$$
Then,
$$
\int_\o Div\,\sigma\cdot\rr =
-\int_\o \sigma:D\rr
=-\int_\o Curl\vv : D\rr
=\int_\o Div\,Curl\,\vv\cdot\rr=0
$$
for every $\rr\in C_0^\infty(\o)^2$ and therefore $Div\,\sigma=0$.

Now, from Lemma \ref{rotor} we know that there exists
$\ww\in W^{1,p}(\o,1-\alpha)^2$ such that
\begin{eqnarray}
\label{casopart6}
Curl\,\ww=\sigma\hspace*{1cm}{\rm
and}\hspace*{1cm}\|\ww\|_{W^{1,p}(\o,1-\alpha)}
\le C\|\sigma\|_{L^p(\o,1-\alpha)}.
\end{eqnarray}

We have to check that $\di\ww=0$, but
since $\sigma$ is a symmetric tensor we have
\begin{eqnarray*}
\label{casopart5}
\di\ww=\,\frac{\partial w_1}{\partial x_1}
+\frac{\partial w_2}{\partial x_2}
=-\sigma_{12}+\sigma_{21}=0.
\end{eqnarray*}

To conclude the proof observe that in view of (\ref{casopart2}),
(\ref{casopart4}) and (\ref{casopart6}) we have
(\ref{casopart3}) which together with (\ref{casopart2})
yields (\ref{pesoalaizq}). \end{proof}

Now, it is natural to ask whether part or all the weight in the estimate
(\ref{pesoalaizq}) can be moved to the right hand side. We will give a positive
answer
to this question. As we will show, the proof of this more general result is
similar to that of Theorem \ref{divergencia} but it requires some non-trivial
preliminary results. In particular, we will need an extra hypothesis on the
domain.

\begin{defi}
For $0<m\le n$, a compact $F\subset\R^n$
is an ${\mbox{\bf m-set}}$, if there exists a positive constant $C$ such that
$$
C^{-1}r^m < {\mathcal H}^m(B(x,r)\cap F)< C r^m,
$$
for every $x\in F$ and $0<r\le diam F$,
where ${\mathcal H}^m$ is the $m$-dimensional Hausdorff measure
and $B(x,r)$ is the ball with radius $r$ and center $x$.
\end{defi}

The reader who is not familiar with Hausdorff measure can think in
the particular case that $\partial\o$ is a rectifiable curve in $\R^2$ and
$m=1$.
In that case ${\mathcal H}^1$ is the length.

We are going to use that Calder\'on-Zygmund singular integral
operators are continuous in weighted $L^p$-norms, $1<p<\infty$,
for weights in the Muckenhaupt class $A_p$. This is a well known result
which can be seen for example in the book \cite{S2}.

We state and prove the following lemma in the general n-dimensional case
since it does not make any difference with the particular case $n=2$.
Our lemma generalizes the results proved in \cite{DST} for smooth domains.
Since the proof is too technical we postpone
it for an appendix and continue now with our main results.
In what follows we consider the distance to the boundary of $\o\subset\R^n$,
$d(x)$ defined for every $x\in\R^n$ and not only for $x\in\o$.

\begin{lemma}
\label{Ap}
Let $\o\subset\R^n$ be a bounded domain such that
its boundary is an ${\mbox{\bf m-set}}$. If
$-(n-m)<\mu<(n-m)(p-1)$, then $d^\mu$ belongs to the class
$A_p$.
\end{lemma}
\begin{proof} See Apendix. \end{proof}

As a consequence we have the following result on weighted estimates for
solutions of the divergence.

\begin{lemma}
\label{divAp1}
Let $\o\subset\R^2$ be a bounded domain such that its boundary
is a ${\mbox{\bf 1-set}}$. Given $f\in L^p(\o,\g)$, $1<p<\infty$,
with $-1/p<\g\le 1- 1/p$
there exists $\vv\in W^{1,p}(\o,\g)^2$ such that
$$
\di\vv=f
$$
and
$$
\|\vv\|_{W^{1,p}(\o,\g)^2}
\le C\|f\|_{L^p(\o,\g)}
$$
\end{lemma}

\begin{proof} Extend $f$ by zero to $\R^2$. Then,
it is well known that
$$
\phi(x)=-\frac1{2\pi}\int_{\R^n}\log\,|x-y|\,f(y)\,dy
$$
is a solution of $\Delta\phi=f$. Moreover, it follows from
the theory of singular integral operators (see for example \cite{S2})
that, if $w\in A_p$,
$$
\int_{\R^2}\left|\frac{\partial^2\phi(x)}{\partial x_i\partial x_j}\right|^p\,
w(x) \, dx
\le\int_{\R^2} |f(x)|^p \,w(x) \, dx.
$$
But, since $\mu=\g p$ satisfies the hypothesis
of Lemma \ref{Ap} with $n=2$ and $m=1$, $d^{\mu}\in A_p$
and therefore $\vv:=\nabla\phi$
is the desired solution. \end{proof}

We can now give our more general result on solutions
of the divergence.

\begin{theorem}
\label{divergencia}
Let $\o\subset\R^2$ be a bounded domain such that its boundary
is a ${\mbox{\bf 1-set}}$. Given $f\in L_0^p(\o,\b-1)$, $1<p<\infty$,
if $\a\le\b\le 1$ and $-1/p<\b-1$, there exists
$\uu\in W_{const}^{1,p}(\o,\beta-\alpha)^2$ such that
$$
\di\uu=f
$$
and
\begin{equation}
\label{est1}
\|D\uu\|_{L^p(\o,\beta-\alpha)}
\le C\|f\|_{L^p(\o,\beta-1)}
\end{equation}
\end{theorem}

\begin{proof} Since $-1/p<\b-1$, it follows from
Lemma \ref{divAp1} that there exists
$\vv\in W^{1,p}(\o,\beta-1)^2$ such that
\begin{equation}
\label{parte6}
\di\vv=f
\end{equation}
and
\begin{equation}
\label{parte6,5}
\|\vv\|_{W^{1,p}(\o,\beta-1)}
\le C\|\vv\|_{W^{1,p}(\o,\beta-\alpha)}
\le C\|f\|_{L^p(\o,\beta-1)}.
\end{equation}

The rest of the proof follows as that of Theorem \ref{divergencia-part}.
Now we have to show that there exists $\ww\in W^{1,p}(\o,\beta-\alpha)^2$
satisfying $\di\ww=0$ and such that
$$
\vv-\ww \in W_{const}^{1,p}(\o,\beta-\alpha)^2
$$
and
$$
\|D\ww\|_{L^p(\o,\beta-\alpha)}
\le C\|f\|_{L^p(\o,\beta-1)}.
$$
The reader can easily check that the existence of $\ww$ follows by using
Lemma \ref{tensorsimetrico} as
in Theorem \ref{divergencia}. \end{proof}

\section{Domains with external cusps}
\setcounter{equation}{0}
\label{external cusps}

In this section we consider the particular case
of the H\"older-$\alpha$ domain defined as
\begin{eqnarray}
\label{ej}
\o=\Big{\{}(x,y)\in\R^2\,:\,0<x<1\,
,\,0<|y|<x^{1/\alpha}\Big{\}}
\end{eqnarray}
with $0<\alpha\le 1$.

We are going to show that in this case
the weaker boundary condition imposed in
Theorem \ref{divergencia} is equivalent to the standard one, i.e., that
the solution of the divergence obtained in that theorem can be modified, by
adding a constant vector field, to obtain a solution which vanishes on the
boundary
in the classic sense.

We will consider the particular case $\b=\a$ of our general
Theorem \ref{divergencia}. Extension of the arguments to other cases
might be possible but it is not straightforward.

\begin{theorem}
\label{divergencia2}
Let $\o\subset\R^2$ be the domain
defined in (\ref{ej}) and $1<p<\infty$. If $1-1/p<\a\le 1$ then,
given $f\in L_0^p(\o,\a-1)$ there exists
$\uu\in W^{1,p}_0(\o)^2$ such that
\begin{equation}
\label{cusp1}
\di\uu=f
\end{equation}
and
\begin{equation}
\label{cusp2}
\|\uu\|_{W^{1,p}_0(\o)}
\le C\|f\|_{L^p(\o,\a-1)}
\end{equation}
with a constant depending only on $p$ and $\a$.
\end{theorem}

\begin{proof} It is easy to see that $\o$ satisfies the hypotheses of
Theorem \ref{divergencia}. Therefore, it follows from that theorem
that there exists $\uu\in W_{const}^{1,p}(\o)^2$ which verifies (\ref{cusp1}).

We are going to prove that, for any
$\psi\in W_{const}^{1,p}(\o)$, there exists a constant $\psi_0\in\R$
such that
$$
\psi-\psi_0\in W^{1,p}_0(\o):=\overline{C_0^\infty(\o)}.
$$
Consequently, $\uu$ can be modified by adding a constant
to each of its components to obtain the desired solution. Indeed,
the estimate (\ref{cusp2}) will follow form (\ref{est1}) by the Poincar\'e
inequality.

Given $\psi\in W_{const}^{1,p}(\o)$, let us show first that $\psi$
is constant on $\partial\o$.
From the definition of $W_{const}^{1,p}(\o)$ we have that
$$
\int_\o \cu\psi\cdot\nabla\phi=0 \qquad \forall \phi\in W^{1,p'}(\o).
$$

Now, let $(x_0,y_0)$ be a point in $\partial\o$ different
from the origin and $B$ an open ball centered in $(x_0,y_0)$ such
that $0\notin B$. Taking $\phi\in C^\infty(B)$ we have

$$
0=\int_\o \cu\psi\cdot\nabla\phi
=-\int_{B\cap\partial\o} \psi\frac{\partial\phi}{\partial t}
\qquad \forall \phi\in C^\infty(B)
$$
where $\frac{\partial\phi}{\partial t}$ indicates the tangential derivative
of $\phi$.
Consequently $\frac{\partial\psi}{\partial t}=0$ in the distributional sense
on $B\cap\partial\o$ and then, since
$\partial\o-(0,0)$ is a connected set, we conclude that there exists a constant
$\psi_0$ such that $\psi=\psi_0$ on $\partial\o$. To simplify notation
we assume in what follows that $\psi_0=0$ and so, we have to see
that $\psi\in W^{1,p}_0(\o)$.

Now, let $\zeta\in C^\infty(\R_+)$ be such that
$$
\zeta\equiv 1\ {\rm in}\ [0,1]\hspace*{1cm}\zeta\equiv 0\ {\rm
in}{\mbox\ \R_+-(0,2)}\hspace*{1cm}0\leq\zeta\leq 1.
$$

We decompose $\psi$ as
$$
\psi(x,y)=\zeta(3x)\psi(x,y)+\left(1-\zeta(3x)\right)\psi(x,y)
=: \psi_1+\psi_2.
$$
It is easy to see that $\psi_2\in W_0^{1,p}(\o_2)$
where $\o_2$ is the Lipschitz domain
$$
\o_2\,:=\, \o\cap\Big{\{}x> \frac{1}{3}\Big{\}}.
$$

Thus, we can suppose that $\psi=\psi_1$.
Let now $\phi_n\in C^\infty(\o)$ be a sequence satisfying
$\phi_n\rightarrow \psi$ in $W^{1,p}(\o)$ and let
$\gamma:=1/\alpha$.

It is easy to check that, for $y\in(0,1)$,
$$
|\phi_n(x,x^{\gamma}-y)|\le |\phi_n(x,x^\gamma)|
+ \int_0^y\left|
\frac{\partial \phi_n}{\partial y}(x,x^{\gamma}-t)\right|\,dt.
$$

Therefore, integrating and using the H\"older inequality we have
\begin{eqnarray*}
\int_{y^\alpha}^1|\phi_n(x,x^\gamma-y)|^p\,dx
\leq C
\left(\int_{y^\alpha}^1|\phi_n(x,x^\gamma)|^p\,dx\,
+ y^{p-1}\int_{y^\alpha}^1\int_0^y\left|
\frac{\partial \phi_n}{\partial y}(x,x^{\gamma}-t)\right|^p\,dt\,dx\right).
\end{eqnarray*}

Thus, using the continuity of the trace in the Lipschitz
domain $\o\cap \{x>y^\alpha\}$ we have

\begin{eqnarray}
\label{T1}
\int_{y^\alpha}^1|\psi(x,x^\gamma-y)|^p\,dx&=&
\lim_{n\rightarrow\infty}\int_{y^\alpha}^1|\phi_n(x,x^\gamma-y)|^p\,dx\nonumber\\
&\leq&C\lim_{n\rightarrow\infty}\left(\int_{y^\alpha}^1|\phi_n(x,x^\gamma)|^p\,dx
\,+\,y^{p-1}\int_{y^\alpha}^1\int_0^y\left| \frac{\partial
\phi_n}{\partial y}(x,x^{\gamma}-t)\right|^p\,dt\,dx\right)\nonumber\\
&=&C\,y^{p-1}\int_{y^\alpha}^1\int_0^y\left|
\frac{\partial\psi}{\partial y}(x,x^{\gamma}-t)\right|^p\,dt\,dx.
\end{eqnarray}

Now we will show that the sequence $\psi_m$ defined by
$$
\psi_m(x,y):= \psi(x,y)\left(1-\zeta_m(x^\gamma-|y|)\right),
$$
where $\zeta_m(t):=\zeta(mt)$, converges to $\psi$ in
$W^{1,p}(\o)$. Moreover, it is easy to see that
$\mbox{supp\,}\psi_m\subset\o$.

By symmetry we can assume that $\o=\o\cap\{y>0\}$.
Using the dominated convergence theorem we obtain
$$
\lim_{m\rightarrow \infty}\|\psi-\psi_m\|_{L^p(\o)}^p
=\lim_{m\rightarrow\infty}
\int_\o\left|\psi(x,y)\zeta_m(x^\gamma-y)\right|^p=0.
$$
On the other hand,
$$
\frac{\partial \psi_m}{\partial x}(x,y)
=\frac{\partial\psi}{\partial x}(x,y)
\Big{(}1-\zeta_m(x^\gamma-y)\Big{)}-m\,\psi(x,y)\,\zeta^\prime\gamma
x^{\gamma-1}
$$
and then,
\begin{eqnarray*}
\int_\o\left|\frac{\partial\psi}{\partial x}
-\frac{\partial\psi_m}{\partial x}\right|^p \,
&\leq&\,\int_\o\left|\frac{\partial\psi}{\partial
x}(x,y)\,\zeta_m(x^\gamma-y)\right|^p
+ C
m^p\int_\o\left|\psi(x,y)\chi_{\{y>\psi(x)-2/m\}}\right|^p\\\\\,&=:&I\,+\,II.
\end{eqnarray*}

Thus, using again dominated convergence, it is easy to check that
$I\rightarrow 0$. So, it only remains to analyze $II$.

Now, by the change of variables defined by
$(x,y)\longmapsto(x,x^\gamma-y)$ and using (\ref{T1}) it
follows that
\begin{eqnarray*}
II\,&=&\,C\,m^p\,\int_0^{2/m}\int_{y^\alpha}^1|\psi(x,x^\gamma-y)|^p\,dx\,dy\,\\
&\leq&\,\,C\,m^p\,\int_0^{2/m}\,y^{p-1}\int_{y^\alpha}^1\int_0^y\left|
\frac{\partial\psi}{\partial y}(x,x^{\gamma}-t)\right|^p\,dt\,dx\,dy\\
&\leq&\,\,C\,m^p\,\int_0^{2/m}\,y^{p-1}\int_0^{2/m}\int_{t^\alpha}^1\left|
\frac{\partial \psi}{\partial y}(x,x^{\gamma}-t)\right|^p\,dx\,dt\,dy\\
&\leq&\,\,C\,m^p\,\left(\frac{2}{m}\right)^p\int_0^{2/m}\int_{t^\alpha}^1\left|
\frac{\partial \psi}{\partial y}(x,x^{\gamma}-t)\right|^p\,dx\,dt\\
&\leq&\,C\int_\o\left|\frac{\partial\psi}{\partial y}(x,y)
\chi_{\{y>\psi(x)-2/m\}}\right|^p\,\longrightarrow 0
\end{eqnarray*}

\bs An analogous argument can be applied to prove that
$\frac{\partial\psi_m}{\partial y}\to\frac{\partial \psi}{\partial y}$
in $L^p(\o)$.

Consequently, we conclude the proof by observing that $\psi_m$
belongs to $W^{1,p}_0(\o)$. \end{proof}

In the following theorem we show that the estimate (\ref{cusp2}) is
optimal in the sense that it is not possible to improve the power of
the distance in the right hand side. Recall that $p'=\frac{p}{p-1}$ is
the dual exponent of $p$.

\begin{theorem}
Let $\o$ be the domain defined in (\ref{ej}).
If $div:W_0^{1,p}(\o)^2\rightarrow L_0^p(\o,\b)$ admits
a continuous right inverse for some $\b\le 0$ then, $\b \le \a-1$.
\end{theorem}
\begin{proof}
For $s<\frac{1-\b p'+\a}{\a p'}$ define
$f_s(x,y)=x^{-\frac{s}{p-1}}d(x,y)^{-p'\b}$.
Then, calling $\o_+=\o\cap \{y>0\}$, we have
$$
\|f_s\|_{L^p(\o,\beta)}^p
= 2 \int_{\o_+} x^{-sp'} d(x,y)^{-\b pp'+\b p}\, dx dy
= 2\int_{\o_+} x^{-sp'} d(x,y)^{-\b p'}\, dx dy
$$
and therefore, using that for $y>0$,
$d(x,y)\simeq x^{1/\alpha}-y$, we obtain
$$
\|f_s\|_{L^p(\o,\beta)}^p
\simeq 2\int_{\o_+} x^{-sp'} (x^{1/\alpha}-y)^{-\b p'}\, dx dy
$$
but,
$$
\int_{\o_+} x^{-sp'} (x^{1/\alpha}-y)^{-\b p'}\, dx dy
=\int_0^1\int_0^{x^{\frac1{\a}}} x^{-sp'} (x^{1/\alpha}-y)^{-\b p'}\, dy dx
$$
$$
=\frac1{1-\b p'}\int_0^1 x^{-sp'} x^{(1-\b p')/\a}\,dx
=\frac1{1-\b p'}\,\frac{1}{p'(\frac{1-\b p'+\a}{\a p'} - s)}
$$
where we have used $s<\frac{1-\b p'+\a}{\a p'}$.
Therefore,
\begin{equation}
\label{fs}
\|f_s\|_{L^p(\o,\beta)}^p
\simeq\frac1{A-s}
\end{equation}
where $A:=\frac{1-\b p'+\a}{\a p'}$ and with constants in the equivalence
independent of $s$.

Now, let $B$ be a ball such that $\overline B\subset\o$ and
$\omega\in C_0^\infty(B)$ such that $\int_B\omega=1$.
From our hypothesis we know that, if $c_s=\int_\o f_s$, there exists
$\vv_s\in W_0^1(\o)^2$
such that
$$
\di\vv_s=f_s-c_s\omega\hspace*{1cm}{\rm and}\hspace*{1cm}
\|\vv_s\|_{W_1^p(\o)}\le C\|f_s-c_s\omega\|_{L^p(\o,\b)}.
$$
But, since $\b\le 0$,
\begin{equation}
\label{cotacs}
|c_s|= \|f_s\|_{L^1(\o)}\le C \|f_s\|_{L^p(\o,\b)}
\end{equation}
and so,
\begin{equation}
\label{cotavs}
\|\vv_s\|_{W_1^p(\o)}\le C\|f_s\|_{L^p(\o,\b)}
\end{equation}
where we have used that $\|\omega\|_{L^p(\o,\b)}\le C$ because
the support of $\omega$ is contained in $B$.
Then,
\begin{eqnarray*}
\label{optimo1}
\|f_s\|^p_{L^p(\o,\b)}&=&\int_\o f^{p-1}_s\,(f_s-c_s\omega)\,d^{p\b}
+ \int_\o f^{p-1}_s\,c_s\omega\,d^{p\b}\nonumber\\
&=&\int_\o f^{p-1}_s \di\vv_s\,d^{p\b}+ \int_\o
f^{p-1}_s\,c_s\omega\,d^{p\b}\\
&=&\int_{\o}x^{-s}\,\di\vv_s+\int_\o
f^{p-1}_s\,c_s\omega\,d^{p\b}.\nonumber\\
\end{eqnarray*}
Using again that the support of $\omega$ is at a positive distance from the
boundary, together with (\ref{cotacs}), it follows that
$$
\int_\o f^{p-1}_s\,c_s\omega\,d^{p\b}
\le C\|f_s\|_{L^p(\o,\b)}.
$$
On the other hand,
\begin{eqnarray*}
\int_{\o}x^{-s}\,\di\vv_s&=&s\int_{\o}x^{-s-1}\,\vv_{s,1}
=s\int_{\o}\frac{\partial(y\,x^{-s-1})}{\partial y}\vv_{s,1}\\
&=&-s\int_{\o}y\,x^{-s-1}\frac{\partial \vv_{s,1}}{\partial y}\,
\le s\|y\,x^{-s-1}\|_{L^{p'}(\o)}\,\|\vv_s\|_{W_1^p(\o)}\\
&\le&Cs\|y\,x^{-s-1}\|_{L^{p'}(\o)}\,\|f_s\|_{L^p(\o,\b)}\\
\end{eqnarray*}
where for the last inequality we have used (\ref{cotavs}).

Therefore,
\begin{eqnarray}
\label{optimo2} \|f_s\|^{p-1}_{L^p(\o,\b)}
\le C\{s\|y\,x^{-s-1}\|_{L^{p'}(\o)}+1\}
\end{eqnarray}
But, an elementary computation shows that
\begin{equation}
\label{fs2}
\|y\,x^{-s-1}\|^{p'}_{L^{p'}(\o)}
\simeq\frac1{B-s}
\end{equation}
where $B:=\frac{1- (\a-1) p'+\a}{\a p'}$ and with constants in the equivalence
independent of $s$.

Thus, from (\ref{fs}), (\ref{optimo2}) and
(\ref{fs2}) we conclude that there exists a constant
independent of $s$ such that
$$
\frac1{A-s}\le C\frac1{B-s}
$$
therefore, $B\le A$ and it follows immediately that $\b\le\a-1$.
\end{proof}

\begin{remark}
If we put some part of the weight in the left hand side as in (\ref{est1}),
it is possible to prove a more general result, namely, under some restriction
on the exponents in the weights, the
difference between the powers in the right and left sides cannot be
less than $1-\a$ (see \cite{ADL2}).
\end{remark}

\section{An application to the Stokes equations}
\setcounter{equation}{0}
\label{application}

In this section we show how our results
can be applied to the analysis of the Stokes equations when
$\o$ is the domain defined in (\ref{ej}).

We are going to use the well known theory developed by Brezzi (see for example
\cite{BF,D2,GR}) but modifying the usual Hilbert spaces and the bilinear form
corresponding to the divergence free restriction in the weak formulation
of the Stokes equations.

\begin{theorem}
Let $\o$ be the domain defined in (\ref{ej}) with
$1/2<\a \le 1$. Then, if $\ff\in H^{-1}(\o)^2$, there exists a unique
weak solution $(u,p)\in H^1_0(\o)^2 \times L_0^2(\o,1-\a)$
of the Stokes equations (\ref{stokes}). Moreover, there exists a constant C
depending only on $\a$ such that
\begin{equation}
\label{cota solucion}
\|\vv\|_{H^1_0(\o)} + \|p\|_{L^2(\o,1-\a)}
\le C \|\ff\|_{H^{-1}(\o)}.
\end{equation}
\end{theorem}
\begin{proof}
Let us introduce the spaces
$$
V=\Big\{\vv\in H^1_0(\o)^2\,:\,\di\vv\in L^2(\o,\alpha-1)\Big\}
$$
which is a Hilbert space with the norm
$$
\|\vv\|^2_V:=\|\vv\|^2_{H_0^1(\o)}+\|\di\vv\|^2_{L^2(\o,\a-1)},
$$
and
$$
Q=L^2_0(\o,\alpha-1).
$$
Define the bilinear forms $a:V\times V\rightarrow \R$
and $b:V\times Q\rightarrow \R$ by
$$
a(\uu,\vv)=\int_\o D\uu:D\vv
$$
and
$$
b(\vv,q)=\int_\o \di\vv\,q\,d^{2\a-2}.
$$
We are going to show that the problem
\begin{eqnarray}
\label{stokes2}
a(\uu,\vv) + b(\vv,q)&=&\int_\o f\cdot \vv \qquad \forall\vv\in V\\
b(\uu,r)\,\qquad\qquad &=&0 \qquad \qquad\ \ \forall r\in Q
\label{stokes3}
\end{eqnarray}
has a unique solution $(\uu,q)\in V\times Q$.

Using the Schwarz inequality it is easy to check that
the bilinear forms $a$ and $b$ are continuous and, since
$\ff\in H^{-1}(\o)^2$,
that the linear functional defined by the right hand side of (\ref{stokes2})
is continuous.

Let
$$
W=\Big\{\vv\in V\,:\,b(\vv,r)=0\ \forall r\in Q\Big\}.
$$
According to Brezzi's theory
it is enough to see that $a$ is coercive in $W$ and $b$ satisfies
the inf-sup condition
\begin{equation}
\label{infsup}
\inf_{r\in Q}\ \sup_{\vv\in V}
\frac{b(\vv,r)}{\|r\|_Q\,\|\vv\|_V} > 0
\end{equation}

Since $\di V\subset Q$ we can take $r=\di\vv$ in the equation $b(\vv,r)=0$
and conclude that $W=\{\vv\in H^1_0(\o)^2\,:\,\di\vv=0\}$. Therefore,
coerciveness of $a$ in $W$ follows from the Poincar\'e inequality.

On the other hand, from Theorem \ref{divergencia2} we know that
given $r\in L^2_0(\o,\alpha-1)$ there exists $\ww\in H_0^1(\o)$
such that
$$
\di\ww=r
\qquad \mbox{and} \qquad
\|\ww\|_{H_0^1(\o)}\le C\|r\|_{L^2_0(\o,\alpha-1)}
$$
where $C$ is a positive constant which depends only on $\a$.
Moreover, from the definition of the norm in $V$ it follows immediately
that
$$
\|\ww\|_V\le C_1\|r\|_Q
$$
for another constant depending only on $\a$.
Then,
$$
\sup_{\vv\in V}
\frac{b(\vv,r)}{\|r\|_Q\,\|\vv\|_V}
\ge \frac{\int_\o \di\ww\,r\,d^{2\alpha-2}}{\|r\|_Q\,\|\ww\|_V}
=\frac{\|r\|_Q}{\|\ww\|_V}\ge C_1^{-1}
$$
and therefore the inf-sup condition (\ref{infsup}) is proved.

Summing up we have proved that the problem given in (\ref{stokes2}) and
(\ref{stokes3}) has a unique solution $(\uu,q)\in V\times Q$. Moreover,
it follows also from the general theory that there exists a constant
$C$ depending only on $C_1$ such that
\begin{equation}
\label{cota solucion2}
\|\vv\|_V + \|q\|_Q
\le C \|\ff\|_{H^{-1}(\o)}.
\end{equation}

Now, define $p=q\,d^{2\a-2}$. It easy to see that $p\in L^2(\o,1-\alpha)$ and
moreover, it follows from (\ref{stokes3}) that $\di\uu=0$ and from
(\ref{stokes2}) that $(\uu,p)$ verifies
$$
\int_\o D\uu:D\vv - \int_\o \di\vv\,p=0 \qquad \forall\vv\in V.
$$
Therefore, since $C_0^\infty(\o)\subset V$, $(\uu,p)$ is a solution
of the Stokes equations (\ref{stokes}) in the sense of distributions
as we wanted to prove. Finally, since $\|p\|_{L^2(\o,1-\alpha)}=\|q\|_Q$,
(\ref{cota solucion}) follows immediately
from (\ref{cota solucion2}).
\end{proof}

We end this section with a corollary which gives an estimate for the pressure
in a standard $L^r$-norm.

\begin{corollary}
Let $\o$ be the domain defined in (\ref{ej}) with
$1/2<\a \le 1$ and $(\uu,p)\in H^1_0(\o)^2 \times L_0^2(\o,1-\a)$
be the solution of the Stokes equations (\ref{stokes}).
If $\ff\in H^{-1}(\o)^2$ and $1\le r<2/(3-2\a)$ then
$(\uu,p)\in H^1_0(\o)^2 \times L^r(\o)$.
Moreover, there exists a constant C depending only on $\a$ such that
$$
\|\uu\|_{H^1_0(\o)} + \|p\|_{L^r(\o)}
\le C \|\ff\|_{H^{-1}(\o)}
$$
\end{corollary}

\begin{proof} We only have to prove that $p\in L^r(\o)$ and
that
\begin{equation}
\label{cotapLr}
\|p\|_{L^r(\o)}\le C \|\ff\|_{H^{-1}(\o)}.
\end{equation}

Observe that $\int_o d^{\b} < +\infty$ for any $\b>-1$. Indeed,
this follows easily by using that $d(x,y)\simeq x^{1/\alpha}-|y|$.
Then, applying the H\"older inequality with exponent $2/r$, we have
$$
\|p\|^r_{L^r(\o)}=\int_\o |p|^r d^{(1-\a)r} d^{(\a-1)r}
\le \|p\|^r_{L^2(\o,1-\alpha)} \left(\int_\o
d^{\frac{2(\a-1)r}{2-r}}\right)^{\frac{2-r}2}
$$
but the integral in the right hand side is finite because
$(2(\a-1)r)/(2-r)>-1$.
So $\|p\|_{L^r(\o)}\le C \|p\|_{L^2(\o,1-\alpha)}$ and therefore,
(\ref{cotapLr})
follows immediately from (\ref{cota solucion}).
\end{proof}

\section{Appendix}
\setcounter{equation}{0}
\label{appendix}

To prove Lemma \ref{Ap} we will work with Whitney decompositions.
If $F$ is a compact non-empty subset of $\R^n$, then
$\R^n\setminus F$ can be represented as a union of closed dyadic
cubes with pairwise disjoint interior $Q^k_j$ satisfying
\begin{eqnarray}
\label{whitney}
\R^n\setminus F\,=\,\bigcup_{k\in\Z}\bigcup_{j=1}^{N_k}Q^k_j
\end{eqnarray}
where the edge length of $Q^k_j$ is $2^{-k}$. The decomposition
(\ref{whitney}) is called a {\it Whitney decomposition of
$\R^n\setminus F$} and the collection $\{Q_j^k\,:\,j=1,...,N_k\}$ is called the
$k^{th}$ generation of Whitney cubes. Furthermore, the Whitney cubes
satisfy
$$
\ell_k\le d(Q_j^k,F)\le 4 \ell_k
$$
where $d(Q_j^k,F)$ denotes the distance of the cube to $F$ and
$\ell_k$ the diameter of $Q_j^k$ (see for example \cite{S}).

For $x_0\in F$ and $R>0$, $N_k\left( B(x_0,R) \right)$ denotes the number of
Whitney cubes
of $F^c$ in the $k^{th}$ generation contained in $B(x_0,R)$.

\begin{lemma}
\label{m-set} Let $F\subset\R^n$ be a compact $\mbox{m-set}$.
Given $x_0\in F$ and $0<R<diam(F)/3$, there exists a constant $C$
depending only on $F$ such that

$$
N_k\left(B(x_0,R)\right)\,\leq\, C\,R^m\,2^{km}
$$
\end{lemma}

\begin{proof} The idea is to use that the number of Whitney
cubes of $F^c$ in the $k^{th}$ generation
contained in a ball $B$ is
essentially the number of balls of radius $2^{-k}$
necessary to cover $F\cap B$.

Let $Q^k$ be a Whitney cube in the $k^{th}$ generation
contained in $B(x_0,R)$. Then, it is easy to check that
$$d(Q^k,F)=d(Q^k,F\cap B(x_0,2R)).$$

Suppose there exist balls $B(x_i,2^{-k})$ with $x_i\in F$,
for $1\leq i\leq N$, satisfying the following properties
\begin{eqnarray}
\label{bolas} F\cap
B(x_0,2R)\subseteq\bigcup_{i=1}^NB(x_i,2^{-k})\hspace*{1cm}{\rm
and }\hspace*{1cm} N\,\leq\,C\,R^m\,2^{km}.
\end{eqnarray}

Thus, if $y_Q\in F$ is a point satisfying $d(Q^k,F)=d(Q^k,y_Q)$ we
can conclude that there is $x_i$, for some $1\leq i\leq N$, such
that $y_Q\in B(x_i,2^{-k}).$

So, using that $Q^k$ is a Whitney cube in the $k^{th}$ generation
it follows that
$$
Q^k\subset B(x_i,6\ell_k),
$$
where $\ell_k$ is the diameter of $Q^k$. But, $B(x_i,6\ell_k)$ cannot
contain more than a finite number $c(n)$ of Whitney cubes $Q^k$.
Then, by (\ref{bolas}) it follows that
$$
N_k\left( B(x_0,R) \right)\,
\le c(n)N \le C\,R^m\,2^{km}.
$$
Thus, to complete the proof we have to show (\ref{bolas}).
Let $r=2^{-k}$.
For $F_0:=F\cap B(x_0,2R)$ we define the numbers
$$
H_m(F_0,r):=\min\Big{\{}Nr^m\,:\,F_0\subseteq\bigcup_{i=1}^N
B(x_i,r)\Big{\}}
$$
and
$$
P(F_0,r):=\max\Big{\{}N\,:\,{\rm there\ exists\ disjoint\ balls\
}B(x_i,r),\,i=1,\dots, N,
{\rm with}\ x_i\in F_0\Big{\}}.
$$

Then, using that $F$ is an $\mbox{m-set}$ we have
\begin{eqnarray*}
H_m(F_0,r)\,&\leq&\,P\left(F_0,\frac{r}{2}\right)r^m\,=\,2^m\,P\left(F_0,\frac{r}{2}\right)
\left(\frac{r}{2}\right)^m\,\\
&<&\,2^m\,C\sum_{i=1}^{P(F_0,r/2)}{\mathcal H}^m\left(B\left(x_i,\frac{r}{2}\right)\cap
F\right)\\
&=&\,2^m\,C\sum_{i=1}^{P(F_0,r/2)}{\mathcal H}^m\left(B\left(x_i,\frac{r}{2}\right)\cap
F\cap B(x_0,3R)\right)\\
&\leq&\,2^m\,C\,{\mathcal H}^m\left(F\cap B(x_0, 3R)\right)\,<\,C^2 6^m
\,R^m.
\end{eqnarray*}

Thus, using the definition of $H_m(F_0,r)$ we obtain
(\ref{bolas}), concluding the proof.
\end{proof}

Before proving Lemma \ref{Ap} let us recall the definition
of the Muckenhaupt class $A_p$. For $1<p<\infty$ a non-negative
function $w$ is in $A_p$ if
\begin{eqnarray}
\sup_{ B\subset\R^n}\left(\frac{1}{|B|}\int_B
w(x)\,dx\right)\left(\frac{1}{|B|}\int_B
w(x)^{-\frac{1}{p-1}}\,dx\right)^{p-1}<\infty.
\end{eqnarray}
where the supremum  is taken over all the balls $B$.

\bigskip

\noindent{\bf Proof of Lemma \ref{Ap}:}
Let $B$ be a ball in $\R^n$, $r_B$ its radius and $d(B)$ the
distance of $B$ to $\partial\o$.

If $r_B\le d(B)$, given $x$ in $B$ we have
$d(B)\le d(x)\le 3d(B).$
Then,
\begin{eqnarray*}
\left(\frac{1}{|B|}\int_B
d^{\mu}\right)\left(\frac{1}{|B|}\int_B
d^{-\frac{\mu}{p-1}}\right)^{p-1}\,\leq
C\,\left(\frac{1}{|B|}\int_B
d(B)^{\mu}\right)\left(\frac{1}{|B|}\int_B
d(B)^{-\frac{\mu}{p-1}}\right)^{p-1}\,\leq C\,
\end{eqnarray*}

On the other hand, if $r_B\ge d(B)$, there exists
$x_0\in\partial\o$ such that $B\subseteq B(x_0,3 r_B)$. Then,
without loss of generality, we can assume that $B$ is centered at a
point of $\partial\o$.

Now, from the Whitney decomposition of ${\partial\o}^c$ we
have
\begin{eqnarray*}
\label{Apprmi} \left(\frac{1}{|B|}\int_B
d^{\mu}\right)\left(\frac{1}{|B|}\int_B
d^{-\frac{\mu}{p-1}}\right)^{p-1}
\le C r_B^{-np}\left(\sum_{Q^k}\int_{Q^k}
d^{\mu}\right)\left(\sum_{Q^k}\int_{Q^k}
d^{-\frac{\mu}{p-1}}\right)^{p-1}=: {\rm I}
\end{eqnarray*}
where the sum is taken over all Whitney cubes $Q^k$ intersecting $B$.
There is no loss
of generality in assuming that the Whitney cubes are contained in $B$.

Observe that if $Q^k$ is contained in $B$ then
$2^{-k}\le\frac{1}{\sqrt{n}}r_B$. We call $k_0$ the
minimum $k$ satisfying this inequality. Then, it is easy to see that
$2^{-k_0}\simeq r_B$.

Now, using that $d(x)\simeq d(Q^k)\simeq 2^{-k}$ for every $x\in
Q^k$ and Lemma \ref{m-set} we obtain

\begin{eqnarray*}
{\rm I}&\le&C r_B^{-np}\left(\sum_{Q_k}
2^{-k\mu}2^{-kn}\right)\left(\sum_{Q_k} 2^{\frac{\mu
k}{p-1}}2^{-kn}\right)^{p-1}\\&\le&C
r_B^{-np}\left(\sum_{k=k_0}^\infty
N_k(B(x_0,r_B))2^{-k\mu}2^{-kn}\right)\left(\sum_{k=k_0}^\infty
N_k(B(x_0,r_B))2^{\frac{\mu k}{p-1}}2^{-kn}\right)^{p-1}\\
&\leq&Cr_B^{-np}\left(\sum_{k=k_0}^\infty
r_B^{m}2^{-k(\mu+n-m)}\right)\left(\sum_{k=k_0}^\infty
r_B^{m}2^{-k\left(n-m-\frac{\mu}{p-1}\right)}\right)^{p-1}={\rm II}.
\end{eqnarray*}
Then, since $-(n-m)<\mu<(p-1)(n-m)$, we obtain

\begin{eqnarray*}
{\rm II}&\le&C\,r_B^{-p(n-m)}\left(2^{-k_0(\mu+n-m)}\right)
\left(2^{-k_0\left(n-m-\frac{\mu}{p-1}+1\right)}\right)^{p-1}
\le C\,r_B^{-p(n-m)}\left(2^{-k_0}\right)^{p(n-m)}
\le C
\end{eqnarray*}
and therefore the Lemma is proved.\Box

\end{document}